\documentclass{amsart}
\usepackage[utf8]{inputenc}
\usepackage{amsmath}
 \usepackage{amsfonts}
 \usepackage{stmaryrd}
\usepackage{amssymb,latexsym}
\usepackage{amsfonts,mathrsfs}
\usepackage[margin=1.8in]{geometry}
\usepackage[all,cmtip]{xy}
\usepackage[colorlinks= true,pdfborder={0 0 0}]{hyperref}
\usepackage{graphicx}
 \usepackage{stmaryrd}
 \usepackage{esint}
 \usepackage {verbatim}
\usepackage{color}
\def\f#1#2{\frac{#1}{#2}}

\def\pa{\partial}
\def\n{\nabla}

\def\a{\alpha}
\def\b{\beta}
\def\ga{\gamma}

\def\({\left (}
\def\){\right )}
\def\<{\langle}
\def\>{\rangle}
\def\A{\mathring{A}}

\newcommand{\bel}[1]{\begin{equation}\label{#1}}

\newcommand{\be}{\begin{equation}}
  \newcommand{\beq}{\begin{equation}}

\newcommand{\ba}{\begin{eqnarray}}
\newcommand{\ea}{\end{eqnarray}}

\newcommand{\qe}{\end{equation}}
\newcommand{\eeq}{\end{equation}}

\newtheorem{thm}{Theorem}[section]

\newtheorem{lem}[thm]{Lemma}
\newtheorem{prop}[thm]{Proposition}

\newtheorem{rem}[thm]{Remark}
\newtheorem{claim}{Claim}[section]
\newtheorem*{acknowledgement*}{Acknowledgement}

\newcommand{\abs}[1]{\left\vert#1\right\vert}
\newcommand{\set}[1]{\left\{#1\right\}}

\newcommand{\eps}{\varepsilon}

\newcommand{\half}{\frac{1}{2}}

\title[stability of the curvature flow]{Stability of the area preserving mean curvature flow in Asymptotic Schwarzschild space}
\keywords{area preserving mean curvature flow; stability; long time existence, exponential convergence; asymptotic schwarzschild space}
\author{Yaoting Gui, Yuqiao Li, Jun Sun}
\address{Yaoting Gui, Beijing International Mathematical Research Center, Peking University}
\address{Yuqiao Li, Department of Mathematics, Hefei University of Technology, Hefei, 230009, P.R.China}
\address{Jun Sun, School of Mathematics and Statistics, Wuhan University, Wuhan, 430072, P. R. of China}
\email{ytgui@bicmr.pku.edu.cn, lyq112@mail.ustc.edu.cn, sunjun@whu.edu.cn}
\thanks {2020 Mathematics Subject Classification. 51F99,31E05}
\thanks{The third author is supported by NSFC No. 12071352, 12271039  }

\begin{document}

%===================================================================================================
\numberwithin{equation}{section}

\begin{abstract}
     We first demonstrate that the area preserving mean curvature flow of hypersurfaces in space forms exists for all time and converges exponentially fast to a round sphere if the integral of the traceless second fundamental form is sufficiently small. Then we show that from sufficiently large initial coordinate sphere, the area preserving mean curvature flow exists for all time and converges exponentially fast to a constant mean curvature surface in 3-dimensional asymptotically Schwarzschild spaces. This provides a new approach to the existence of foliation established by Huisken and Yau (\cite{huisken1996definition}). And also a uniqueness result follows.
\end{abstract}

\maketitle

\section{Introduction}

\allowdisplaybreaks

Let $M_0$ be a smooth compact hypersurface of $\mathbb{R}^{n+1}$ locally represented by an embedding $F_0:U\rightarrow\mathbb{R}^{n+1}$ for $U\subset\mathbb{R}^{n}$ and $F_0(U)\subset M_0$. The volume preserving mean curvature flow was considered by Gage for $n=1$ (\cite{Gage-1986}) and Huisken for $n\geq 2$ (\cite{Huisken-VPMCF}), which is a family of maps $F_t=F(\cdot, t)$ evolving by
\begin{equation}\label{flow-VPMCF}
\begin{cases}   
    \frac{\pa F}{\pa t}(x, t)&=[h(t)-H(x,t)]\nu(x,t),\\
    F(\cdot,0)&=F_0,
\end{cases}
\end{equation}
where $H$ is the mean curvature of $M_t=F_t(U),\ \nu$ is the outward unit normal to $M_t$, and
\begin{equation}\label{h}
h(t)=\frac{\int_{M_t}Hd\mu_t}{\int_{M_t}d\mu_t}
\end{equation}
is the average of the mean curvature.

This flow (\ref{flow-VPMCF})  preserves the volume enclosed by $M_t$ while decreases the area of $M_t$, even when the ambient manifold is a general Riemannian manifold. Huisken  (\cite{Huisken-VPMCF}) demonstrated that if the initial hypersurface $M_0$ is uniformly convex in ${\mathbb R}^{n+1}$ for $n\geq 2$, then the flow (\ref{flow-VPMCF}) exists globally and converges exponentially to a round sphere. This result also provides an alternative proof of the isoperimetric inequality for uniformly convex domains in Euclidean space. Intuitively, if the flow (\ref{flow-VPMCF}) exists globally and converges to some limit hypersurface at infinity, then the limit must be a hypersurface of constant mean curvature. Therefore, the flow (\ref{flow-VPMCF}) can also be used to investigate the existence of constant mean curvature hypersurfaces.

One key challenge in studying the flow (\ref{flow-VPMCF}) is its non-local nature, as $h(t)$ is a global term. This means the maximum principle for parabolic equations is somewhat subtle to apply. Consequently, some elegant properties of the mean curvature flow (\cite{Huisken-InventMath-1986}) are lost. For instance, in general, the embedding property may not be maintained along the flow (\ref{flow-VPMCF}). Additionally, the convexity of the initial hypersurface may not be preserved during the evolution in a general ambient manifold even in space form, as was pointed out by Huisken (\cite{Huisken-VPMCF}). Nevertheless, several extensions of Huisken’s results on the volume-preserving mean curvature flow (\ref{flow-VPMCF}) still hold.

In their seminal paper (\cite{huisken1996definition}), Huisken and Yau investigate the flow (\ref{flow-VPMCF}) in an asymptotically Schwarzschild 3-space. They demonstrated that the flow originating from a sufficiently large coordinate sphere exists globally and converges exponentially to a surface of constant mean curvature. This framework allowed them to establish the existence of constant mean curvature surface foliation and to define the center of mass

Further stability results for the flow (\ref{flow-VPMCF}) have been obtained without the assumption of initial hypersurface convexity. Based on center manifold analysis, Escher and Simonett (\cite{Escher-Simonett-1998}) showed that if the initial hypersurface $M_0$ is $C^{1,\b}$-close to a fixed Euclidean sphere yet non-convex, the flow converges exponentially fast to a round sphere. For a general
ambient manifold, Alikakos and Freire (\cite{AlikakosFreire-JDG-2003}) proved long time existence and
convergence to a constant mean curvature surface, provided that the initial hypersurface is $C^{2,\a}$-close to a sufficiently small geodesic sphere and meets certain non-degeneracy conditions. Li (\cite{Li2009TheVM}) applied iteration techniques from Ricci flow to show the long-term existence and exponential convergence of the flow (\ref{flow-VPMCF}) in Euclidean space under the assumption that the $L^2$-norm of the traceless second fundamental form of the initial hypersurface is sufficiently small. These results were later generalized to different cases where the ambient manifold are space forms(\cite{XuHW2014spaceform}, \cite{Huang-Lin-Zhang}, \cite{LXZ-IJM-2024}).

Another generalization is the higher order mixed volume preserving curvature flow as considered by McCoy in \cite{McCoy2004TheMV}, which is a family of maps $F_t=F(\cdot, t)$ evolving by
\begin{equation}\label{flow}
\begin{cases}   
    \frac{\pa F}{\pa t}(x, t)&=[h_k(t)-H(x,t)]\nu(x,t),\\
    F(\cdot,0)&=F_0,
\end{cases}
\end{equation}
where for each fixed $k=-1,0,1,\cdots,n-1$,
\begin{equation}\label{hk}
h_k(t)=\frac{\int_{M_t}HE_{k+1}d\mu_t}{\int_{M_t}E_{k+1}d\mu_t}.
\end{equation}
For any $l=0,\cdots,n$, $E_l$ is the l-th
elementary symmetric function of the principal curvatures $\kappa_1,\cdots,\kappa_n$ of $M_t$,
$$
E_l=
\begin{cases}
1,&l=0,\\
\sum_{1\leq i_1<...<i_l\leq n}\kappa_{i_1}\kappa_{i_2}\cdots\kappa_{i_l},&l=1,\dots,n.
\end{cases}
$$
For each fixed $k$, the global function $h_k(t)$ in \eqref{hk} ensures that the mixed volume 
\begin{equation}\label{vol}
V_{n-k}(\Phi_t)=
\begin{cases}
\operatorname{Vol}\left(\Phi_t\right),&k=-1,\\\left\{\left(n+1\right)\binom{n}{k}\right\}^{-1}\int_{M_t}E_kd\mu_t,&k=0,1,\dots,n-1,
\end{cases}
\end{equation}
is preserved under the flow \eqref{flow}. Here $\Phi_t$ is the $(n+1)-$dimensional region enclosed by $M_t$ with $\partial\Phi_t=M_t$. 

McCoy has shown that the flow \eqref{flow} has long time solution and converges to a sphere for any $k=-1, 0,1,\cdots, n-1,$ provided that the initial hypersurface $M_0$ is strictly convex \cite{McCoy2004TheMV}. Notably, the flow \eqref{flow} reduces to the flow \eqref{flow-VPMCF} when $k=-1$ indicating that McCoy's result serves as a higher order generalization of Huisken's result. Subsequently, McCoy examined more mixed volume preserving curvature flows (\cite{McCoy-CVPDE2005}).

\vspace{.1in}

In this paper, we will address the flow \eqref{flow} with $k=0$. Set
\[
h_0=\f{\int_{M_t}H^2d\mu_t}{\int_{M_t}Hd\mu_t},
\]
and the flow becomes
\begin{equation}\label{flow0}
\begin{cases}   
    \frac{\pa F}{\pa t}(x, t)&=[h_0(t)-H(x,t)]\nu(x,t),\\
    F(\cdot,0)&=F_0.
\end{cases}
\end{equation}
Due to the natural choice of the global term $h_0$, the flow (\ref{flow0}) preserves the area of $M_t$ while increases the volume of the region enclosed by $M_t$ (see Section 2).  We refer to this flow (\ref{flow0}) as \emph{area preserving mean curvature flow}.

\vspace{.1in}

In the first part of the paper, we establish the stability result of the area preserving mean curvature flow (\ref{flow0}) under the analogous conditions presented in \cite{Li2009TheVM}. Our first theorem is as following:

\begin{thm}\label{thm1.1}
   Let $M\subset{\mathbb R}^{n+1}$ be a compact, orientable, smoothly immersed hypersurface of dimension $n\geq2$ satisfying 
    \begin{equation}\label{init}
        \operatorname{Vol}(M)\leq v,\quad \abs{A}\leq\Lambda,\quad H\geq\ga, \quad\int_{M}|\A|^2\leq\eps
    \end{equation}
   for any positive constant $v, \Lambda, \ga$ and sufficiently small $\eps=\eps(v, \Lambda, \ga)>0$, where Vol is the volume of $M$ with respect to the induced metric and $\A$ is the traceless second fundamental form of $M$.
   Then the flow \eqref{flow0} with initial hypersurface $M$ will stay for all the time and converge exponentially fast to a round sphere.   
\end{thm}

Arguing in a similar way as in \cite{XuHW2014spaceform} and \cite{Huang-Lin-Zhang} for the volume preserving mean curvature flow (\ref{flow-VPMCF}),  Theorem \ref{thm1.1} can be readily generalized to the cases where the ambient manifold is either sphere ${\mathbb S}^{n+1}$ or hyperbolic space $\mathbb{H}^{n+1}$.

\begin{thm}\label{thm1.2}
   Let $M\subset\mathbb{S}^{n+1}$ be a compact, orientable, smoothly immersed hypersurface of dimension $n\geq2$ satisfying initial condition \eqref{init}.
   Then the flow \eqref{flow0} with initial hypersurface $M$ will stay for all the time and converge exponentially fast to a round sphere.   
\end{thm}

\begin{thm}\label{thm1.3}
   Let $M\subset\mathbb{H}^{n+1}$ be a compact, orientable, smoothly immersed hypersurface of dimension $n\geq2$ satisfying 
    \begin{equation*}
        \operatorname{Vol}(M)\leq v,\quad \abs{A}\leq\Lambda,\quad H\geq\ga+n, \quad\int_{M}|\A|^2\leq\eps.
    \end{equation*}
   Then the flow \eqref{flow0} with initial hypersurface $M$ will stay for all the time and converge exponentially fast to a round sphere.   
\end{thm}

\begin{rem}
  Note that the initial condition is the same as the Euclidean space when the ambient space is the sphere, while on the hyperbolic space, we add the condition $H\geq\ga+n$, which is based on the fact that any constant mean curvature hypersurface in the hyperbolic space has mean curvature greater than $n$.  
\end{rem}

In the second part of this paper, we will explore the flow (\ref{flow0}) in a 3-dimensional asymptotically Schwarzschild space. Similar to the findings of Huisken-Yau in \cite{huisken1996definition}, we show that the flow \eqref{flow0} starting from a coordinate sphere with sufficiently large radius exists globally and converges exponentially to a surface of constant mean curvature.

\begin{thm}\label{thm-HY}
    Let $N$ be a 3-dimensional asymptotically Schwarzschild space with metric
    \[ \bar{g}_{\alpha\beta}=\left(1+\frac{m}{2r}\right)^4\delta_{\alpha\beta}+P_{\alpha\beta}, \]
    where
    \[ |P_{\alpha\beta}|\leq C_1r^{-2},\quad|\partial^lP_{\alpha\beta}|\leq C_{l+1}r^{-l-2},\quad 1\leq l\leq 4, \]
    and $m>0$ is a constant. Set $C_0=\max(1, m, C_1, C_2, C_3, C_4, C_5)$. Let $M_0\subset N$ be the coordinate sphere of large radius $\sigma>0$. Then there is a $\sigma_0$ depending only on $C_0$ such that for all $\sigma\geq\sigma_0$, the flow \eqref{flow0} starting from $M_0$ exists for all time and converges exponentially fast to a constant mean curvature surface. Moreover, this surface is unique in the circular domain $\mathcal{B}_{\sigma}(B_1, B_2, B_3)$.
\end{thm}

\begin{rem}
    The constant m can be interpreted
as the total mass of an isolated system in the end N and is referred to as ADM–mass in the physical literature, comparing \cite{MR127946}, and see also \cite{Bartnik1986TheMO} for a geometric definition. The famous positive mass theorem states that m is non-negative, provided the scalar curvature of N is non-negative and must be positive unless N is flat \cite{MR526976,MR612249}. 
\end{rem}

\begin{rem}
    In Theorem 4.1 of \cite{huisken1996definition}, Huisken and Yau established the uniqueness result of the constant mean curvature surface $M_\sigma$ within the circular domain $\mathcal{B}_{\sigma}(B_1, B_2, B_3)$, which they derived by using the volume-preserving mean curvature flow (\ref{flow-VPMCF}). They also demonstrated that each $M_{\sigma}$ is strictly stable and form a proper foliation of $N\backslash B_{\sigma_0}(0)$. Furthermore, they proved that the center of gravity of $M_{\sigma}$ has a limit as $\sigma\to\infty$, defining this limit as the "center of mass". In Theorem 5.1 of \cite{huisken1996definition}, they proved uniqueness of the foliation outside of some large metric ball with the radius depending on the mean curvature. And subsequently, Qing and Tian (\cite{QingTian}) provided a complete uniqueness of the foliation outside of some fixed compact set.
    Hence, the constant mean curvature surface we construct in Theorem \ref{thm-HY} is exactly the same as that of Huisken and Yau. In other words, we provide a new approach to construct the foliation of the 3-dimensional asymptotically Schwarzschild space and also the construction of “center of mass”.
\end{rem}
In a forthcoming paper, we will investigate a different but still related geometric flow, which allows us to construct another foliation in 3-dimensional asymptotically Schwarzschild space. This foliation is geometrically different from the one constructed by Huisken and Yau and also a different geometric center of mass. This may be of independent interest. 

\vspace{.1in}

The proof of Theorem \ref{thm1.1} employs an iteration argument, which is inspired by Li's approach for volume preserving mean curvature flow (\ref{flow-VPMCF}). For the proof of Theorem \ref{thm-HY}, we adopt the conceptions from Huisken and Yau (\cite{huisken1996definition}) to show that the surfaces will remain within some fixed circular domain with nice curvature estimates along the flow and then the long time existence and convergence follow naturally. However, several new challenges arise when transfering from the volume preserving mean curvature flow (\ref{flow-VPMCF}) to the area preserving mean curvature flow (\ref{flow0}).

Firstly, along the volume preserving mean curvature flow \eqref{flow-VPMCF}, $\int_{M_t}(h-H)d\mu_t$ vanishes identically, which simplifies the derivation of bounds of $h$. In constrast, in our case, where the $\int_{M_t}(h_0-H)d\mu_t$ is merely non-negative, we are unable to achieve the lower bounds of $h_0$ directly. Instead, we utilize the lower bound of $h(t)$ as the proxy of $h_0(t)$ along the area preserving mean curvature flow \eqref{flow0}, which as a consequence allows us to derive the upper bound of $h_0(t)$.

Secondly, in the proof of the exponential convergence part of Theorem \ref{thm-HY}, we encounter a challenge when trying to estimate the integral of $(h_0-H)^2$ as previously done by Huisken and Yau for the flow (\ref{flow-VPMCF}). This difficulty arises from certain limitations of the flow \eqref{flow0}. To overcome this difficulty, we shift our focus to estimating the quantity $\frac{h_0}{h}$. By analyzing this ratio, we aim to derive the desired exponential decay result. This strategy allows us to bypass the complications associated with the integral estimation while still achieving the convergence of the flow.

The following sections of this paper are organized as follows. In Section 2, we recall some known formulas that will be essential for later discussions. In Section 3, we give the proof of the stability result of the area preserving mean curvature flow under the smallness assumption on the integral of the traceless  second fundamental form. Finally in Section 4, we will prove the long time existence  and exponential convergence of the flow (\ref{flow0}) in a 3-dimensional asymptotically Schwarzschild space starting from a coordinate sphere with sufficiently large radius.

\begin{acknowledgement*}
The authors would like to express their sincere gratitude to professor Jiayu Li for constant encouragement. The first author would also like to thank professor Gang Tian for useful discussions. 
\end{acknowledgement*}

\vspace{.1in}

\section{Notations and evolution equations}

We will use the notations as in \cite{McCoy2004TheMV}. In particular, for a closed hypersurface $M_t\subset \mathbb{F}^{n+1}(c)$, where $\mathbb{F}^{n+1}(c)$ is an $(n+1)$-dimensional space form of constant curvature $c$, let $g=\{g_{ij}\}$  be the induced metric on $M_t$ and $A=\{h_{ij}\}$ denote the second fundamental form of $M_t$. The mean curvature and the traceless second fundamental form are given by
\[ H=g^{ij}h_{ij}, \]
\[ \A_{ij}=h_{ij}-\frac{H}{n}g_{ij}. \]
We need the evolution equations of $M_t$ along the flow \eqref{flow0}.
\begin{lem}(\cite{Miglioranza2020TheVP})\label{evo}
  Suppose the ambient space is a space form $\mathbb{F}^{n+1}(c)$, where $c\in\{1, 0, -1\}$ is its curvature. Along the flow \eqref{flow0}, we have 
    \begin{align*}
        \partial_tg_{ij}=&2(h_0-H)h_{ij},\\
        \partial_td\mu_t=&H(h_0-H)d\mu_t,\\
        \partial_tH=&\Delta H+(|A|^2+nc)(H-h_0),\\
        \pa_t|A|^{2}=&\Delta|A|^{2}-2|\nabla A|^{2}+2|A|^{4}-2h_{0}tr(A^{3})-2ch_0H-2nc|A|^2+4cH^2,\\
\partial_{t}|\mathring{A}|^{2}=&\Delta|\mathring{A}|^{2}-2|\nabla\mathring{A}|^{2}+2|A|^{2}|\mathring{A}|^{2}-2h_{0}\left(tr(\mathring{A}^3)+\frac{2}{n}H|\mathring{A}|^{2}\right)-2nc|\A|^2,\\
\partial_t|\nabla\A|^2=&\Delta|\nabla\A|^2-2|\nabla^2\A|^2+\nabla A*\nabla \A*A*\A+\nabla\A*\nabla\A*A*A\\
&+\nabla\A*\A*A*\bar{Ric}+\nabla\A*\A*\bar{Ric},
    \end{align*}
    where $A*B$ denotes the contraction of the tensors $A$ and $B$ and $\bar{Ric}$ is the Ricci curvature tensor of the ambient space.
\end{lem}
 
By the first variation of the area, we know the flow \eqref{flow0} preserves 
$\int_{M_t}E_0d\mu_t=|M_t|$ the area of $M_t$. By the H\"older inequality, we have
\begin{equation}\label{vol-evo}
    \frac{d}{dt}\text{Vol}(\Phi_t)=\int_{M_t}(h_0-H)d\mu_t= \frac{\int_{M_t}H^2d\mu_t}{\int_{M_t}Hd\mu_t}|M_t|-\int_{M_t}Hd\mu_t\geq 0, 
\end{equation}
provided that $\int_{M_t}Hd\mu_t>0$. Therefore, we see that the flow \eqref{flow0} preserves the area of $M_t$ and does not decrease the volume enclosed by $M_t$ and we call it the area preserving mean curvature flow.

\vspace{.1in}

\section{Area preserving mean curvature flow in space forms}

In this section, we will prove the long time existence and convergence of the flow \eqref{flow0} in space forms under the assumption of smallness of the traceless second fundamental form. Our primary focus will be to prove \ref{thm1.1} for the flow in $\mathbb{R}^{n+1}$ using the iteration argument akin to that presented in \cite{Li2009TheVM}.

We first show the short time existence in $\mathbb{R}^{n+1}(n\geq2)$.
% Recall Li's strategy of the convergence result for the volume preserving mean curvature flow. 

% \begin{prop}[Proposition 13 in \cite{Li2009TheVM}]
%  Suppose that the solution $M_t$ satisfies the inequalities
% $$
% |A|(0)\leq\Lambda,\quad h\geq\gamma>0,\quad
% \int_M|\mathring{A}|(0)^2\leq\eps.
% $$
% Then there exists a $T=T(\Lambda)$ such that $M_t$ satisfies
% \[
% |A|(t)\leq2\Lambda\quad and\quad h\geq\gamma/2\quad for\:t\in[0,T]
% \]
% and for fixed $\tau>0$ there exists some constant $C_1=C_1(n,\tau,\Lambda,\upsilon)$ such that 
% \begin{gather*}
% |I-h|^2(t)\leq C_1(n,\tau,\Lambda,v)\eps^{1/2}\\
% |\mathring{A}|(t)\leq C_1(n,\tau,\Lambda,v)\eps^{1/2}\quad for\:t\in[\tau,T].
% \end{gather*}
% \end{prop}
% In the new setting, we have 
\begin{prop}\label{prop1.2}
     Suppose that the hypersurface $M_0$ satisfies
$$
|A|(0)\leq\Lambda,\quad H(0)\geq\gamma>0,\quad
\int_{M_0}|\mathring{A}|^2(0)\leq\eps.
$$
Then there exists a positive $T=T(\Lambda, \ga, n)$ such that $M_t$ satisfies
\[
|A|(t)\leq2\Lambda\quad and\quad H(t)\geq\gamma/2\quad \text{for} \:t\in[0,T],
\]
and for fixed $\tau>0$, there exists some constant $C_1=C_1(n,\tau,\Lambda,\ga)$ such that 
\begin{gather*}
|H-h_0|^2(t)\leq C_1(n,\tau,\Lambda,\ga)\eps^{1/2},\\
|\mathring{A}|(t)\leq C_1(n,\tau,\Lambda,\ga)\eps^{1/2},\quad for\:t\in[\tau,T].
\end{gather*}
\end{prop}
\begin{proof}
    From Lemma \ref{evo}, the evolution equation for $\abs{A}^2$ is
    \[
\pa_t|A|^{2}=\Delta|A|^{2}-2|\nabla A|^{2}+2|A|^{4}-2h_{0}tr(A^{3}),
\]
which implies
\[
\partial_t|A|\leq\Delta|A|+|A|^3+h_{0}|A|^{2} .
\]
Here we used the inequality (see Lemma 2.2 of \cite{Huisken-Sinestrari-CVPDE1999})
\[
|tr(A^{3})|\leq |A|^3.
\]
We use the method in \cite{Miglioranza2020TheVP} to show the short time existence. Let
\[
\mathcal{S}=\left\{\tau>0\mid H(t)\geq\frac{\ga}{2},\quad|A|(t)\leq2\Lambda ,\quad\forall t\in[0,\tau]\right\}.
\]
We claim that
\begin{claim}
    $T=\sup \mathcal{S}>0$.
\end{claim}
Indeed, for $t\in[0, T]$, we have
\[
h_0(t)=\frac{\int_{M_t} H^2d\mu_t}{\int_{M_t} Hd\mu_t}\leq\max_{M_t} H(x,t)\leq\sqrt{n}\max_{M_t}|A|(x,t)\leq2\Lambda\sqrt{n}.
\]
Let $f(t)=\max_{M_t}\abs{A}$,
then there holds
\[
\partial_{t}f\leq f^{3}+\sqrt{n}f^{3}\leq(\sqrt{n}+1)(2\Lambda)^{3}.
\]
This implies 
\[
f(t)\leq(\sqrt{n}+1)(2\Lambda)^{3}t+\Lambda.\]
Taking
\[(\sqrt{n}+1)(2\Lambda)^{3}t+\Lambda\leq2\Lambda,\]
we get
\[t\leq\frac{1}{8(\sqrt{n}+1)\Lambda^{2}}:=T_{1}.
\]
Then, for $t\in[0, T_1]$, we have $|A|(t)\leq2\Lambda$.
Note also that, for $t\in[0, T]$,
\[
\partial_{t}H=\Delta H+|A|^{2}(H-h_{0})\geq\Delta H-|A|^2h_0.
\]
Since $|A|^2h_0\leq\sqrt{n}(2\Lambda)^3$, we have
\[
\pa_{t}H=\Delta H+|A|^{2}(H-h_{0})\geq\Delta H-\sqrt{n}(2\Lambda)^3.
\]
Thus
\[
\pa_{t}\min_{M_t}H\geq-\sqrt{n}(2\Lambda)^3.
\]
It follows
\[
\min_{M_t}H\geq\min_{M_0}H-\sqrt{n}(2\Lambda)^3t\geq \ga-\sqrt{n}(2\Lambda)^3t.
\]
By taking
\[ \ga-\sqrt{n}(2\Lambda)^3t\geq\frac{\ga}{2}, \]
we obtain that
\[
t\leq\frac{\ga}{16\sqrt{n}\Lambda^3}:=T_{2}.
\]
Let $T=\min\{T_{1},T_{2}\}$, then $ T=T(\Lambda,\gamma,n)>0$, which proves the claim. 

Now on $[0,T]$, we see that
\begin{align*}
&\pa_t\int_{M_t}|\mathring{A}|^{2}d\mu_{t}\\
 =&\int_{M_t}\left[\Delta|\mathring{A}|^{2}-2|\nabla \A|^{2}+2|A|^{2}|\mathring{A}|^{2}-2h_{0}(tr(\A)^{3}+\frac{2}{n}H|\A|^{2})+|\A|^2H(h_{0}-H)\right]d\mu_{t} \\
\leq&2(2\Lambda)^{2}\int_{M_{t}}|\A|^{2}d\mu_t+2\sqrt{n}(2\Lambda)^{2}\int_{M_{t}}|\A|^{2}d\mu_t+2n(2\Lambda)^{2}\int_{M_{t}}|\A|^{2}d\mu_{t} \\ 
=&C(n)\Lambda^{2}\int_{M_t}|\mathring{A}|^{2}d\mu_t.
\end{align*}
We get 
\[
\int_{M_t}|\A|^{2}d\mu_{t}\leq e^{C(n)\Lambda^{2}t}\eps, \quad\text{for}\quad t\in[0,T].
\]
Now the evolution equation for $|\mathring{A}|^2$ gives 
\begin{align*}
\partial_{t}|\mathring{A}|^{2}&=\Delta|\mathring{A}|^{2}-2|\nabla\mathring{A}|^{2}+2|A|^{2}|\mathring{A}|^{2}-2h_{0}\left(tr(\mathring{A})^{3}+\frac{2}{n}H|\mathring{A}|^{2}\right)\notag \\
&\leq\Delta|\mathring{A}|^{2}+C(n)\Lambda^{2}|\mathring{A}|^{2}\quad  \text{for}\quad t\in[0,T].
\end{align*}
The Moser iteration then implies
\begin{align*}
|\mathring{A}|(t) & \leq C(n,\Lambda,\tau)\left(\int_{t-\tau}^{t}\int_{M_s}|\mathring{A}|^{2}d\mu_sds\right)^{\frac12}\notag\\
& \leq C(n,\Lambda,\tau)\varepsilon^{\frac{1}{2}}e^{C(n)\Lambda^{2}t}\notag\\
& \leq C_{1}(n,\Lambda,\tau,\ga) \varepsilon^{\frac12},\quad for\, t\in[\tau,T].
\end{align*}
By gradient estimates and the fact that $|A|\leq 2\Lambda$, we have
\begin{equation}\label{gra}
   |\nabla A|\leq C(\tau,\Lambda,\gamma),\quad|\nabla^{2}A|\leq C(\tau,\Lambda,\gamma),\quad t\in[\tau,T]. 
\end{equation}
Integrating by parts and Schwartz inequality imply that
\[
\int_{M_t}|\n\mathring{A}|^{2}d\mu_t\leq\left(\int_{M_t}|\mathring{A}|^{2}\right)^{\frac{1}{2}}\left(\int_{M_t}|\Delta\mathring{A}|^{2}\right)^{\frac{1}{2}}\leq C\left(\tau,\Lambda,\gamma, v\right)\varepsilon^{\frac{1}{2}}e^{\frac{1}{2}C(n)\Lambda^{2}t}. 
\]
By Lemma \ref{evo} and \eqref{gra}, there holds 
\[
\partial_t|\nabla\mathring{A}|^{2}\leq\Delta|\nabla\mathring{A}|+C(\tau,\Lambda)|\A|+\Lambda^{2}|\nabla\mathring{A}|^2,\quad\text{for}\quad t\in[\tau, T]. \]
Using the Moser iteration, we have
\[ |\nabla\A|(t)\leq C(\tau, \Lambda, \ga, v, n)\varepsilon^{\frac{1}{4}}, \quad\text{for}\quad t\in [2\tau, T]. \]
By Lemma 1.3 in \cite{huisken1996definition} and Lemma 12 in \cite{Li2009TheVM}, we get for $t\in [2\tau, T]$ that
\[
|H-h_{0}|(t)\leq diam(M_t)\max(|\n H|)\leq C(n,\Lambda,v,\ga)\max(|\n \mathring{A}|)\leq C_{1}(\tau,\Lambda,v, n,\gamma)\varepsilon^{\frac{1}{4}}. 
\]
\end{proof}

The key of the long time existence and exponential convergence of the flow is the following lemma which shows that the traceless second fundamental form decays exponentially fast.
 
\begin{lem}\label{lem1.3}
    For any $\ga>0$, $\exists \eps_0=\eps_0(n,\ga)$ such that if $M_t$ satisfies 
    \[
    H(t)\geq\ga>0,\quad |\mathring{A}|(t)\leq \eps,\quad \abs{H-h_0}^2(t)\leq \eps,
    \]
    for $t\in[0, T]$ and $\eps\in(0, \eps_0)$.
    Then 
    \[
    |\mathring{A}|(x, t)\leq e^{-\a t}\max_{M_0}|\mathring{A}|(0),\quad \text{for\ }t\in[0,T],
    \]
    where $\a=\f{\ga^2}{4n}$.
\end{lem}
\begin{proof}
    From Lemma \ref{evo}, the evolution equation arises
    \[
\partial_{t}|\A|^{2}=\Delta|\A|^{2}-2|\nabla \A|^{2}+2|A|^{2}|\A|^{2}-2h_{0}\big((tr\A^3)+\frac{2}{n}H|\A|^{2}\big).
\]
Note that
\begin{align*}
&-2h_{0}(tr(\A^3)+\frac{2}{n}H|\A|^{2})\\
\leq&2h_{0}|\A|^{3}-\frac{4}{n}h_{0}H|\A|^{2}\\
\leq&2H|\A|^{3}+2|h_{0}-H||\A|^{3}-\frac{4}{n}H^{2}|\A|^{2}+\frac{4}{n}H|H-h_{0}||\A|^{2}.
\end{align*} 
By the Cauchy inequality, there hold
\[ 2H|\A|^3\leq \frac{1}{2n}H^2|\A|^2+2n|\A|^4, \]
\[ \frac{4}{n}H|H-h_{0}||\A|^{2}\leq \frac{1}{2n}H^2|\A|^2+2n\left(\frac{2}{n}\right)^2|H-h_0|^2|\A|^2. \]
Then
\begin{align*}
&2\abs{A}^2|\A|^2-2h_0(tr(\A^3)+\f2n H|\A|^2)\\
\leq&2|\A|^{4}-\frac{2}{n}H^{2}|\A|^{2}+2H|\A|^{3}+2|h_{0}-H||\A|^{3}+\frac{4}{n}H|H-h_{0}||\A|^{2}\\
%\leq\frac{1}{2n}H^{2}|A|^{2}+2n|A|^{4}\leq\frac{1}{2n}H|A|^{2}+2n\cdot(\frac{2}{n})^{2}|H-h_{0}|^{2}|A|^{2} \\
\leq&-\frac{1}{n}H^{2}|\A|^{2}+(2n+2)|\A|^{4}+2|h_{0}-H||\A|^{3}+\frac{8}{n}|H-h_{0}|^{2}|\A|^{2} \\
\leq&-\frac{1}{n}H^{2}|\A|^{2}+\left[(2n+2)\eps^{2}+2\eps^{\frac{3}{2}}+\frac{8}{n}\eps\right]|\A|^{2} \\
=&-\frac{1}{n}H^{2}|\mathring{A}|^{2}+c(n)\eps|\mathring{A}|^{2} \\
\leq&-\frac{\gamma^{2}}{n}|\A|^{2}+c(n)\eps|\A|^{2}\\
=&-\left(\f{\ga^2}n-c(n)\eps\right)|\A|^2.
\end{align*}
This implies 
\[
\pa_{t}|\A|^{2}\leq\Delta|\A|^{2}-(\frac{\gamma^{2}}{n}-c(n)\eps)|\A|^{2}.\]
By the maximum principal, we have
\[
\max_{M_{t}}|\A|^{2}(t)\leq e^{-(\frac{\gamma^{2}}{n}-c(n)\varepsilon)t}\max_{M_0}|\A|^{2}(0).\]
Choose $\eps\in(0,\eps_0)$ such that
$\frac{\gamma^{2}}{n}-c(n)\eps\geq\frac{\gamma^{2}}{2n}$, 
we then obtain
\[
|\A(t)|\leq e^{-\frac{\gamma^{2}}{4n}t}\max_{M_{0}}|\A|(0)\leq\varepsilon e^{-\alpha t}.
\]
\end{proof}

Applying the same argument of Lemma 16 in \cite{Li2009TheVM}, we have the following lemma.

\begin{lem}\label{lem1.4}
  Suppose that the solution $M_t$ satisfies
  \[
|A| (t) \leq \Lambda,\quad|\A| (t) \leq \eps e^{-\a t},\quad H(t)>0,\quad \text{for}\quad t\in [0, T].
\]
Then for some $\tau>0$, we have
$$|H-h_0|^2(t)\leq C(n,\tau,\Lambda)\eps e^{-\a t},\quad for\:t\in[\tau,T].
$$
\end{lem}

Finally, we estimate the mean curvature by analysing $h$ and $h_0$ along the flow, which is the key difference from the volume preserving mean curvature flow (\ref{flow-VPMCF}).

% \begin{lem}\label{lem1.5}
% If
% \[
% H(t)>0,\quad|\A|\leq\varepsilon,\quad|H-h_{0}|^{2}(t)\leq\varepsilon,\quad \text{for}\quad t\in[0,T],
% \]
% then
% \[
% H(t)\leq\left[1+\frac{(\max_{M_0}H+\frac{n\eps^2}{\max_{M_0}H})\tan(\frac{\eps^{\f32}t}{\sqrt{n}})}{\sqrt{n}\eps-\max_{M_0}H\tan(\frac{\eps^{\f32}t}{\sqrt{n}})}\right]\max_{M_0}H,
% \]
% \[
% H(t)\geq\left[1-\frac{(\min_{M_0}H-\frac{n\eps^2}{\min_{M_0}H})\tan(\frac{\eps^{\f32}t}{\sqrt{n}})}{\sqrt{n}\eps+\min_{M_0}H\tan(\frac{\eps^{\f32}t}{\sqrt{n}})}\right]\min_{M_0}H,
% \]
% for $t\in[0, T]$.
% \end{lem}
% \begin{proof}
%     By Lemma \ref{evo}, we know that
%     \[ \partial_tH=\Delta H+|A|^2(H-h_0). \]
%     Let $f(t)=\max_{M_t}H(x,t)$, by the maximum principal, there holds
%     \[ \frac{d}{dt} f \leq |A|^2|H-h_0|\leq \eps^{\half}(\eps^2+\frac{f^2}{n})=\frac{\eps^{\half}}{n}(n\eps^2+f^2).\]
%    Then
%            \[ \frac{d}{dt}\arctan\left(\frac{\sqrt{n}\eps}{f}\right)=-\frac{\sqrt{n}\eps}{f^2+n\eps^2}\frac{d}{dt}f
%             \geq -\frac{\eps^{\frac{3}{2}}}{\sqrt{n}}.\]
%     Integrating the above inequality, we obtain
%     \[ \arctan\left(\frac{\sqrt{n}\eps}{f(t)}\right)\geq\arctan\left(\frac{\sqrt{n}\eps}{f(0)}\right)-\frac{\eps^{\frac{3}{2}}}{\sqrt{n}}t,  \]
%     \[ f(t)\leq \left[1+\frac{(f(0)+\frac{n\eps^2}{f(0)})\tan(\frac{\eps^{\f32}t}{\sqrt{n}})}{\sqrt{n}\eps-f(0)\tan(\frac{\eps^{\f32}t}{\sqrt{n}})}\right]f(0). \]
%     Similarly, we can get the lower bound.
% \end{proof}

% Forthermore, we have a uniform bound of $H$ if the traceless second fundamental form is exponentially small.

\begin{lem}\label{lem1.6}
If 
\[
    H(t)>0,\quad|\A|(t)\leq\varepsilon e^{-\a t},\quad|H-h_{0}|^{2}(t)\leq\varepsilon e^{-\a t},\quad\text{for}\quad t\in[0,T],
\]
we have
\[
h(0)e^{-\f\eps\a}-\eps^{\frac{5}{2}}\frac{2}{5\a}e^{\frac\eps\a}\leq h(t)\leq h(0)+\frac{2}{5\a}\eps^{\frac{5}{2}},\quad\text{for}\quad t\in[0,T].
\]
Denote
\[
h(0)e^{-\f\eps\a}-\eps^{\frac{5}{2}}\frac{2}{5\a}e^{\frac\eps\a}=C_{1}(h(0),\a,\eps)>0,
\]
then 
\[
h_0(t)\leq \exp(\frac{2\eps^{\frac{5}{2}}}{5\a C_{1}}+\frac{\eps}{\a n})(h_{0}(0)+\frac{2\eps^3}{3\a C_{1}})=C_{2}(n,h(0),h_{0}(0),\eps,\a)>0. 
\]
Furthermore, we have
\[H(t)=H-h_{0}+h_{0}\leq h_{0}(t)+|H-h_{0}|(t)\leq C_{2}+(\eps e^{-\a t})^{\frac{1}{2}}\leq C_{2}+\eps^{\frac{1}{2}},\]
\[H(t)=h_{0}+H-h_{0}\geq h_{0}(t)-|H-h_{0}|(t)\geq h(t)-\eps^{\frac{1}{2}}\geq C_1-\eps^{\frac{1}{2}}, \]
for $t\in[0, T]$.
\end{lem}

\begin{proof}
Notice that $h(t)=\frac{\int_{M_t}Hd\mu_t}{|M_{t}|}$. We have $|M_{t}|h(t)=\int_{M_{t}}Hd\mu_t$ and from Lemma \ref{evo}, 
\[
\partial_{t}|M_{t}|=\int_{M_{t}}H(h_{0}-H)d\mu_{t}=0.
\]
Direct computation gives
\begin{align*}
|M_{t}|\pa_{t}h(t) &=\pa_{t}\int_{M_{t}}Hd\mu_{t}=\int_{M_{t}}[|A|^{2}(H-h_{0})-H^{2}(H-h_{0})]d\mu_{t} \\
&=\int_{M_{t}}[|\A|^{2}(H-h_{0})+(\frac{1}{n}-1)H^{2}(H-h_{0})]d\mu_{t} \\
&=\int_{M_{t}}[|\A|^{2}(H-h_{0})+(\frac{1}{n}-1)H(H-h_{0})^{2}+(\frac{1}{n}-1)h_{0}H(H-h_{0})]d\mu_{t} \\
&=\int_{M_{t}}[|\A|^{2}(H-h_{0})+(\frac{1}{n}-1)H(H-h_{0})^{2}]d\mu_{t}.
\end{align*}
It follows
\[
\pa_th\leq(\eps e^{-\a t})^{\f52}\Longrightarrow h(t)\leq h(0)+\eps^{\f52}\f2{5\a}.
\]
And 
\[
|M_{t}|\pa_th\geq-(\varepsilon e^{-\a t})^{\frac{5}{2}}|M_{t}|-(\varepsilon e^{-\a t})\int_{M_{t}}Hd\mu_{t} .
\]
Thus
\[
\pa_th\geq-(\eps e^{-\a t})^{\f52}-(\eps e^{-\a t})h,
\]
and
\begin{align*}
\pa_t(he^{-\frac{\eps}{\a}e^{-\a t}})
& =e^{-\frac{\eps}{\a}e^{-\a t}}(\pa_{t}h+\eps e^{-\a t}h)\\
& \geq-(\eps e^{-\a t})^{\f52} e^{-\frac{\eps}{\a}e^{-\a t}}\\
& \geq-\eps^{\f52} e^{-\f52\a t}.
\end{align*}
We then get 
\[
h(t)e^{-\frac{\eps}{\a}e^{-\a t}}-h(0)e^{-\f{\eps}{\a}}\geq-\eps^{\frac{5}{2}}\frac2{5\a},
\]
which implies 
\[
 h(t)\geq  h(0)e^{-\frac{\eps}{\a}}-\eps^{\frac{5}{2}}\frac2{5\a}e^{\frac{\eps}{\a}}=C_{1}(h(0),\varepsilon,\a)>0.
 \]
Therefore, we have $C_1(\eps,\a,h(0))\leq h(t)\leq h(0)+\eps^{\frac{5}{2}}\frac{2}{5\a}$.
Thus 
\[
h_{0}(t)=\frac{\int_{M_t} H^{2}d\mu_t}{\int_{M_t} Hd\mu_t}\geq\frac{\int_{M_t} Hd\mu_t}{|M_{t}|}=h(t)\geq C_{1}(h(0) , \eps,\gamma)>0.
\]

Now for the fact that 
$h_0(t)\cdot\int_{M_t} Hd\mu_t=\int_{M_t} H^{2}d\mu_t$, we have
% \[
% (\pa_{t}h_0)\cdot\int_{M_t} Hd\mu_t=\pa_{t}\int H^{2}d\mu_t- h_0\partial_{t}\int_{M_t} Hd\mu_t.
% \]
\begin{align*}
&\partial_th_{0}\cdot\int_{M_t}Hd\mu_t\\
 =&\partial_{t}\int_{M_t} H^{2}d\mu_t-h_{0}\partial_{t}\int_{M_t} Hd\mu_t \\
=&\int_{M_{t}}[2H\Delta H+2H|A|^{2}(H-h_{0})+H^{3}(h_{0}-H)-h_{0}|A|^{2}(H-h_{0})-h_{0}H^{2}(h_{0}-H)]d\mu_{t} \\
=&\int_{M_{t}}[-2|\n H|^{2}+2H|\A|^{2}(H-h_0)+(\frac{2}{n}-1)H^{3}(H-h_{0})-h_{0}|\A|^{2}(H-h_{0})\\
&\quad\quad-\frac{h_{0}}{n}H^{2}(H-h_{0})-h_{0}H^{2}(h_{0}-H)]d\mu_{t} \\
=&\int_{M_{t}}[-2|\n H|^{2}+2(H-h_{0})^{2}|\A|^{2}+2h_{0}|\A|^{2}(H-h_{0})+(\frac{2}{n}-1)H^{2}(H-h_{0})^{2}\\
&\quad\quad+(\frac{2}{n}-1)H^{2}h_{0}(H-h_{0})-h_{0}|\A|^{2}(H-h_{0})-\frac{1}{n}h_{0}H^{2}(H-h_{0})-h_{0}H^{2}(h_{0}-H)]d\mu_{t}\\
=&\int_{M_{t}}[-2|\n H|^{2}+2|\A|^{2}(H-h_{0})^{2}+h_{0}|\A|^{2}(H-h_{0})+(\frac{2}{n}-1)H^{2}(H-h_{0})^{2}\\
&\quad\quad+\frac{1}{n}h_{0}H^{2}(H-h_{0})]d\mu_{t} \\
=&\int_{M_{t}}[-2|\n H|^{2}+2|\A|^{2}(H-h_{0})^{2}+h_{0}|\A|^{2}(H-h_{0})+(\frac{2}{n}-1)H^{2}(H-h_{0})^{2}\\
&\quad\quad+\frac{1}{n}Hh_{0}(H-h_{0})^{2}]d\mu_{t} \\
\leq&\int_{M_{t}}[2|\A|^{2}(H-h_{0})^{2}+h_{0}|\A|^{2}(H-h_{0})+\frac{1}{n}Hh_{0}(H-h_{0})^{2}]d\mu_{t} \\
\leq& 2(\eps e^{-\gamma t})^{3}|M_{t}|+h_{0}(\eps e^{-\a t})^{\frac{5}{2}}|M_{t}|+\frac{1}{n}h_{0}\eps e^{-\a t}\int_{M_{t}}Hd\mu_{t}.
\end{align*}
This implies 
\begin{align*}
\pa_th_{0}& \leq\frac{2(\eps e^{-\a t})^{3}}{h}+\frac{h_{0}}{h}(\eps e^{-\a t})^{\frac{5}{2}}+\frac{1}{n}h_{0}\eps e^{-\a t} \\
&\leq\frac{2(\eps e^{-\a t})^3}{C_{1}}+\left(\frac{(\eps e^{-\a t})^{\frac52}}{C_{1}}+\frac{\eps e^{-\a t}}{n}\right)h_{0}.
%&\partial th_{0}-(\frac{(\eps e^{-\gamma t})^{\frac{2}{2}}}{C_{1}}+\frac{\eps e^{-\gamma t}}{n})h_{0}\leq\frac{2(\eps e^{-\gamma t})^{3}}{C_{1}}
\end{align*}
Let 
\[
g(t)=\frac{2\eps^{\f52} e^{-\frac{5}{2}\a t}}{5\a C_{1}}+\frac{\eps e^{-\a t}}{\a n}\leq\frac{2\eps^{\frac{5}{2}}}{5\a C_{1}}+\frac{\eps}{\a n}=g(0),
\]
then 
\[ g'(t)=-\left(\frac{(\eps e^{-\a t})^{\frac{5}{2}}}{C_1}+\frac{\eps e^{-\a t}}{n}\right). \]
We find
\[
\partial_{t}(e^{g(t)}h_{0}(t))=e^{g(t)}(\partial_{t}h_{0}+g'(t)h_{0}(t))\leq e^{g(t)}\frac{2(\eps e^{-\a t})^{3}}{C_{1}}\leq e^{g(0)}\frac{2\eps^{3}e^{-3\a t}}{C_{1}}.
\]
which implies 
\begin{gather*}
e^{g(t)}h_{0}(t)\leq e^{g(0)}h_{0}(0)+\frac{2e^{g(0)}\eps^{3}}{3\a C_{1}}=e^{g(0)}\left(h_{0}(0)+\frac{2\eps^{3}}{3\a C_{1}}\right),\\
\Rightarrow h_{0}(t)\leq e^{g(0)}\left(h_{0}(0)+\frac{2\eps^{3}}{3\a C_{1}}\right)=C_{2}(\a,n,\varepsilon,h(0),h_{0}(0)).
\end{gather*}
Therefore, we finally get
\[ H(t)\leq h_{0}(t)+|H-h_{0}|(t)\leq C_2+\eps^{\frac{1}{2}},\]
\[ H(t)\geq C_{1}-\eps^{\frac{1}{2}}.\]

\end{proof}

Now we could prove the long time existence and convergence of the flow \eqref{flow0} by an iteration argument.
We restate the theorem as follows.

\begin{thm}\label{thmr}
    Let $M_0\subset\mathbb{R}^{n+1}$ be a compact, orientable, smoothly immersed hypersurface of dimension $n\geq2$ satisfying 
    \bel{}\label{ini}
    |M_0|\leq v,\quad \abs{A}\leq\Lambda,\quad H\geq\ga,\quad \int_{M_0}|\A|^2\leq\eps,
   \qe
for any positive constants $v, \Lambda, \ga$ and sufficiently small $\eps=\eps(v, \Lambda, \ga)>0$.   
   Then the flow \eqref{flow0} with initial hypersurface $M_0$ will stay for all the time and converge exponentially fast to a round sphere. 
\end{thm}
\begin{proof}
The proof is divided into several steps. 

\textbf{Step 1}:
Suppose the solution satisfies the conditions \eqref{ini} at the initial time. By Proposition \ref{prop1.2}, there exists some positive constant $T_1$ such that 
\[
\abs{A}\leq 2\Lambda,\quad H\geq \f\ga2,\quad \text{for}\quad t\in[0,T_1].
\]
and for some $\tau<\f{T_1}4$, there exists a constant $D_1(n, \tau, \Lambda, \ga)>0$, such that 
\[
\abs{H-h_0}^2\leq D_1\eps^{\f12},\quad |\A|\leq D_1\eps^{\f12}\quad\text{for}\quad t\in[\tau,T_1].
\]
Applying Lemma \ref{lem1.3}, for $t\in[\tau,T_1]$, we choose some $\eps\in(0,\eps_0)$ so that 
 \[
 |\A|(t)\leq e^{-\a(t-\tau)}|\A|(\tau)\leq D_1\eps^{\f12}e^{-\a(t-\tau)}\quad\text{for}\quad t\in[\tau,T_1],
 \]
 where $\a=\frac{\ga^2}{16n}$ and $D_1\eps^{\half}<1$.
By Lemma \ref{lem1.4}, 
 \[
 |H-h_0|^2(t)\leq D_2(n, \tau, 2\Lambda)\eps^{\f12}e^{-\a(t-\tau)}\quad\text{for}\quad t\in[2\tau,T_1],
 \]
 with $D_2\eps^{\half}<1$.
Set $D_3=\max\set{D_1,D_2}$, by Lemma \ref{lem1.6}, we obtain
\begin{align*}
h_{0}(t) & \leq(1+\delta)\cdot h_{0}(2\tau)+(1+\delta)\cdot\frac{4({D_3\eps^{\half}})^3}{3\a \ga}\\
 &\leq(1+\delta)\cdot\sqrt{n}\max_{\Sigma_{2\tau}}\abs{A}+(1+\delta)\cdot\frac{4({D_3\eps^{\half}})^3}{3\a \ga}\\
 & <2(1+\delta)\sqrt{n}\Lambda+(1+\delta)\cdot\frac{4({D_3\eps^{\half}})^3}{3\a \ga}\\
 & <3\sqrt{n}\Lambda+\f12,
 \end{align*}
 for $t\in [2\tau, T_1]$,
where $\delta=\exp[\frac{4(D_3\eps^{\half})^{\frac{5}{2}}}{5\a\ga }+\frac{D_3\eps^{\half}}{\a n}]-1$. We then get 
 \[
H\leq 3\sqrt{n}\Lambda+1\quad\text{for}\quad t\in[2\tau,T_1].
 \]

\textbf{Step 2}: Recall the estimate for $t=T_1-\tau$, 
\[
\abs{A}(t)\leq 2\Lambda,\quad \frac{\ga}{2}\leq H(t)\leq 3\sqrt{n}\Lambda+1,\]
\[|H-h_0|^2\leq D_2\eps^{\f12}e^{-\a(t-\tau)},\quad|\A|\leq D_1\eps^{\f12}e^{-\a(t-\tau)}.
\]
Applying Proposition \ref{prop1.2} once more for the initial time $t=T_1-\tau$ gives with $T_2=T_2(3\sqrt{n}\Lambda+2)$,
\begin{gather*}
  \abs{A}(t)\leq 4\Lambda,\quad H(t)\geq\f\ga4,\\
  |H-h_0|^2(t)\leq D_4\eps^{\f12}e^{-\a(t-\tau)},\quad|\A|(t)\leq D_4\eps^{\f12}e^{-\a(t-\tau)}, 
\end{gather*}
for $t\in[T_1,T_1+T_2]$, where $D_4=D_1(\tau, n, 2\Lambda, \frac{\ga}{2})D_3$.
We now choose $\eps$ small enough that $D_4\eps^{\half}<\eps_0(n, \frac{\ga}{8})$ and apply Lemma \ref{lem1.3} again to see that
\[
|\A|(t)\leq e^{-\a_1(t-\tau)}\max_{M_\tau}|\A|(\tau)\leq D_1\eps^{\f12}e^{-\a_1(t-\tau)}, 
\]
for $t\in[\tau,T_1+T_2]$, where $\a_1=\frac{\ga^2}{64n}<\a$.
By Lemma \ref{lem1.4}, 
 \[
 |H-h_0|^2(t)\leq D_2(n, \tau, 2\Lambda)D_1\eps^{\f12}e^{-\a_1(t-\tau)}\quad\text{for}\quad t\in[2\tau,T_1+T_2],
 \]
 By Lemma \ref{lem1.6}, we have
 \begin{align*}
h_{0}(t) & \leq(1+\delta_1)\cdot h_{0}(2\tau)+(1+\delta_1)\cdot\frac{4({D_3\eps^{\half}})^3}{3\a_1 \ga}\\
 &\leq(1+\delta_1)\cdot\sqrt{n}\abs{A}(2\tau)+(1+\delta_1)\cdot\frac{4({D_3\eps^{\half}})^3}{3\a_1 \ga}\\
 & <2(1+\delta_1)\sqrt{n}\Lambda+(1+\delta_1)\cdot\frac{4({D_3\eps^{\half}})^3}{3\a_1 \ga}\\
 & <3\sqrt{n}\Lambda+\f12,
 \end{align*} 
 for $t\in [2\tau, T_1+T_2]$.
Here $\delta_1=\exp[\frac{4(D_3\eps^{\half})^{\frac{5}{2}}}{5\a_1\ga }+\frac{D_3\eps^{\half}}{\a_1 n}]-1$.
Therefore,
\[
    H\leq h_0+\abs{H-h_0}\leq 3\sqrt{n}\Lambda+1,
\]
\[ |A|(t)\leq |\A|+|H-h_0|+h_0\leq 3\sqrt{n}\Lambda+3, \]
for $t\in[2\tau, T_1+T_2]$.
By Lemma \ref{lem1.4}, 
\[
\abs{H-h_0}^2\leq D_2(n,\tau,3\sqrt{n}\Lambda+3)D_1\eps^{\f12}e^{-\a_1(t-\tau)},\quad \text{for}\quad t\in[2\tau,T_1+T_2].
\]
Set $D_6=\max\set{D_2(n,\tau,3\sqrt{n}\Lambda+3)D_1,D_1}$ and  choose $\eps$ with $D_6\eps^{\half}<1$, then applying Lemma \ref{lem1.6}, we find 
\[
h(t)\geq h(\tau)e^{-\f{D_6\eps^{\half}}{\a_1}}-(D_6\eps^{\half})^{\f52}\f2{5\a_1}e^{\f{D_6\eps^{\half}}{\a_1}},
\]
for $t\in[T_1, T_1+T_2]$, and $H(t)\geq h(t)-\eps\geq \frac{\ga}{3}$.

\textbf{Step 3}: We claim that
\begin{claim}\label{claim}
If the estimates 
\begin{gather*}
    \f\ga3\leq H(t)\leq 3\sqrt{n}\Lambda+1,\\
    |\A|(t)\leq D_6\eps^{\f12}e^{-\a_1(t-\tau)},\\
    \abs{H-h_0}^2(t)\leq D_6\eps^{\f12}e^{-\a_1(t-\tau)}
\end{gather*}
hold for $t\in[2\tau,S]$ and $S\geq T_1+T_2$, then they also hold on $[2\tau,S+T_2]$.
\end{claim}
\begin{proof}
    First on $[2\tau,S]$, 
    \[\abs{A}(t)\leq|\A|(t)+\abs{H}(t)\leq 3\sqrt{n}\Lambda+2=\Lambda_1.\]
     By Proposition \ref{prop1.2} for $t=S-\tau$, there is $T_2(3\sqrt{n}\Lambda+2)$, such that
    \begin{gather*}
  \abs{A}(t)\leq 2\Lambda_1,\quad H(t)\geq\f\ga6,\\
  |H-h_0|^2(t)\leq D_7\eps^{\f12}e^{-\a_1(S-\tau)},\quad|\A|(t)\leq D_7\eps^{\f12}e^{-\a_1(S-\tau)},
\end{gather*}
for $t\in [S, S+T_2]$, where $D_7=D_1(n, \tau, 2\Lambda_1, \frac{\ga}{3})D_6$.
Applying Lemma \ref{lem1.6}, we choose $\eps<\eps_2$ small enough,
\[
H(t)\geq h(t)-\eps\geq h(S)e^{-\f{D_7\eps^{\half}}{\a_1}}-(D_7\eps^{\half})^{\f52}\f2{5\a_1}e^{\f{D_7 \eps^{\half}}{\a_1}} \geq \f\ga4,
\]
for $t\in[S, S+T_2]$.
By Lemma \ref{lem1.3} and \ref{lem1.4}, 
\begin{gather*}
|\A|(t)\leq e^{-\a_1(t-\tau)}\max_{M_{\tau}}|\A|(\tau)\leq D_1\eps^{\f12}e^{-\a_1(t-\tau)},\quad t\in[\tau,S+T_2]\\
\abs{H-h_0}^2(t)\leq D_2D_1\eps^{\f12}e^{-\a_1(t-\tau)}\leq D_6\eps^{\f12}e^{-\a_1(t-\tau)},\quad t\in[2\tau,S+T_2].
\end{gather*}
By Lemma \ref{lem1.6}, we conclude that
\[
H(t)\leq h_0(t)+\eps\leq 3\sqrt{n}\Lambda+1, 
\]
and therefore
\[ |A|(t)\leq \Lambda_1, \]
for $t\in[2\tau,S+T_2].$
And by choosing $\eps<<1$, from Lemma \ref{lem1.6}, we also have 
\[
H(t)\geq h(t)-\eps\geq h(\tau)e^{-\f{D_6\eps^{\half}}{\a_1}}-(D_6\eps^{\half})^{\f52}\f2{5\a_1}e^{\f{D_6\eps^{\half}}{\a_1}}\geq \f\ga3,\]
for $t\in[T_1,T_1+T_2]$.

\end{proof}

\textbf{Step 4}: By Claim \ref{claim}, the flow \eqref{flow0} has long time existence and converges exponentially to a totally umbilical hypersurface, which is the round sphere in $\mathbb{R}^{n+1}$.
  
\end{proof}

From the proof of Theorem \ref{thmr}, we observe that the main components are Proposition \ref{prop1.2}, Lemma \ref{lem1.3}, Lemma \ref{lem1.4} and Lemma \ref{lem1.6}. Thus, it suffices to prove the corresponding lemmas in $\mathbb{S}^{n+1}$ and $\mathbb{H}^{n+1}$, and then the theorem follows in a similar way.

As noted in \cite{XuHW2014spaceform}, we require the Sobolev inequality and diameter estimate for closed hypersurfaces in $\mathbb{S}^{n+1}$ and $\mathbb{H}^{n+1}$. By the evolution equations in Lemma \ref{evo}, the two cases in $\mathbb{S}^{n+1}$ and $\mathbb{H}^{n+1}$ are closely resemble those in the Euclidean space, leading us to omit the intricate details of the proofs.

\vspace{.1in}

\section{ Area preserving mean curvature flow in the asymptotically Schwarzschild spaces}

In this section, we investigate the flow \eqref{flow0} in the asymptotically Schwarzschild spaces. We will prove that the flow \eqref{flow0} exists for all the time and converges exponentially fast to a constant mean curvature surface. This serves as an alternative approach to construct the foliation of constant mean curvature surfaces obtained by Huisken-Yau using the volume preserving mean curvature flow (\ref{flow-VPMCF}) (\cite{huisken1996definition}).

\begin{thm}\label{asy}
    Let $N$ be a 3-dimensional asymptotically Schwarzschild space with metric
    \[ \bar{g}_{\alpha\beta}=\left(1+\frac{m}{2r}\right)^4\delta_{\alpha\beta}+P_{\alpha\beta}, \]
    where
    \[ |P_{\alpha\beta}|\leq C_1r^{-2},\quad|\partial^lP_{\alpha\beta}|\leq C_{l+1}r^{-l-2},\quad 1\leq l\leq 4, \]
    and $m>0$ is a constant. Set $C_0=\max(1, m, C_1, C_2, C_3, C_4, C_5)$.  Let $M_0\subset N$ be the coordinate sphere of large radius $\sigma>0$. Then there is a $\sigma_0$ depending only on $C_0$ such that for all $\sigma\geq\sigma_0$, the flow \eqref{flow0} starting from $M_0$ exists for all time and converges exponentially fast to a constant mean curvature surface.
\end{thm}

\begin{proof}
We prove the theorem by following the approach of Huisken-Yau. First recall the following descriptions for nearly round surfaces proved by Huisken-Yau.

\begin{prop}\label{2.1} (Proposition 2.1 of \cite{huisken1996definition})
Suppose $M$ is a hypersurface in $(N, \bar{g})$ such that $r(y)\geq \frac{1}{10}\max_M r=:r_1$ for all $y\in M$ and such that for some constants $K_1, K_2$
\[
|\A|\leq K_1r_1^{-3},\quad |\nabla\A|\leq K_2r_1^{-4}.
\]
Then there is an absolute constant $c$ such that the curvature $A^e$ of $M$ with respect to the Euclidean metric satisfies
\[
|\A^e|\leq c(K_1+C_0)r_1^{-3},\quad |\nabla^eA^e|\leq c(K_2+C_0)r_1^{-4}.
\]
provided $r_1\geq c(C_0+K_1)$. Moreover, there is a number $r_0\in {\mathbb R}$ and a vector $\vec{a}\in {\mathbb R}^3$ such that
\[
|\lambda_i^e-r_0^{-1}|\leq c(K_1
+K_2+C_0)r_1^{-3},\quad i=1,2,
\]
\[
|(y-\vec{a})-r_0\nu_{e}|\leq c(K_1
+K_2+C_0)r_1^{-1},
\]
\[
|\nu_{e}-r_0^{-1}(y-\vec{a})|\leq c(K_1
+K_2+C_0)r_1^{-2}.
\]
Here $y$ and $\nu_e$ are the position vector and the unit normal of $M$ in ${\mathbb R}^3$, respectively.
\end{prop}

For $\sigma\geq1$ and $B_1, B_2, B_3$ nonnegative numbers, we define a set $\mathcal{B}_{\sigma}(B_1, B_2, B_3)$ of round surfaces in $(N, \bar{g})$ by setting
\[ \mathcal{B}_{\sigma}:=\{ M\subset N|\sigma-B_1\leq r\leq \sigma+B_1, |\A|\leq B_2\sigma^{-3}, |\nabla\A|\leq B_3\sigma^{-4} \}. \]

The next proposition gives us the accuracy information for the mean curvature for surfaces in $\mathcal{B}_{\sigma}(B_1, B_2, B_3)$.

\begin{prop}\label{2.2} (Proposition 2.2 of \cite{huisken1996definition})
Let $M$ be a hypersurface in $\mathcal{B}_{\sigma}(B_1, B_2, B_3)$. Suppose $\sigma\geq c(B_1+B_2+C_0)$ is such that all assumptions of Proposition \ref{2.1} are satisfied
and let $r_0$ and $\vec{a}$ be as in that proposition. Then there is an absolute constant $c$ such that the mean curvature of $M$ satisfies
\[
\left|H-\frac{2}{r_0}+\frac{4m}{r_0^2}-\frac{6m\langle \vec{a}, \nu_e\rangle_e}{r_0^3}\right|\leq c(C_0^2+B_2+B_3)\sigma^{-3},
\]
provided $\sigma\geq c(B_1^2+B_2+C_0)$.
\end{prop}

We next show the cooresponding results of Proposition 3.4 and 3.5 in \cite{huisken1996definition}.

\begin{prop}\label{3.4}
    Suppose that $M_t$ is a smooth solution of the flow \eqref{flow0} which is contained in $\mathcal{B}_{\sigma}(B_1, B_2, B_3)$ for all $t\in [0, T]$. Suppose that $\sigma\geq c(C_0+B_1+B_2)$ is such that Proposition \ref{2.1} applies and $r_0(t)$ is as in that result. Then there is an absolute constant $c$ such that
    \[|r_0(t)-\sigma|\leq c(C_0+B_2+B_3) \]
    holds uniformly in $[0, T]$, provided that $\sigma\geq c(C_0+B_1+B_2)$.
\end{prop}

\begin{proof}
    At time $t=0$, we have $r_0=\sigma$.  By Proposition \ref{2.1}, we have
    \[ \big| |M_t|-|M_t|_e \big|\leq \left| \int_{M_t}(\sqrt{\text{det}g}-1)dx \right|\leq cC_0\sigma, \]
    where $|M_t|_e$ denotes the area of $M_t$ with respect to the Euclidean metric.
    Since the flow \eqref{flow0} preserves the area, there holds
    \[ |M_t|=|M_0|=4\pi\sigma^2+c\sigma. \]
    From Proposition \ref{2.1}, $|M_t|_e=4\pi r_0^2+c\sigma.$
    Therefore, 
    \[ |4\pi\sigma^2-4\pi r_0^2|\leq c(C_0+B_2+B_3)\sigma, \]
    which implies the desired result.
\end{proof}

\begin{prop}
    Suppose that the solution $M_t$ of \eqref{flow0} is contained in $\mathcal{B}_{\sigma}(B_1, B_2, B_3)$ for all $t\in [0, T]$. Then there is an absolute constant $c$ such that
    \[ \max_{M_t}r\leq \sigma+c(m^{-1}+1)(C_0^2+B_2+B_3) \]
     holds uniformly in $[0, T]$, provided that $\sigma\geq c(C_0^2+B_1^2+B_2+B_3)$.
\end{prop}

\begin{proof}
    Let $D>0$ and assume $\max_{M_t}r<\sigma+D$ is violated for the first time at $(y_0, t_0), t_0>0$.
    At that point $r(y_0, t_0)=\max_{M_t}r=\sigma+D$, $\langle y_0, \nu_e\rangle_e=r$ and
   $(h_0-H)\langle \nu, \nu_e\rangle_e\geq0$.
    From Proposition \ref{2.1}, it follows that $\langle \nu, \nu_e\rangle_e\geq\frac{1}{2}$ and $\langle \Vec{a}, \nu_e\rangle_e\geq\frac{1}{2}|\vec{a}|$ at $(y_0, t_0)$ provided $\sigma\geq c(C_0+B_2+B_3)$. 
    So we have $H\leq h_0$ at $(y_0, t_0)$, but we get from Proposition \ref{2.2} that
    \begin{equation}\label{He}
        H=\frac{2}{r_0}-\frac{4m}{r_0^2}+\frac{6m\langle \vec{a}, \nu_e\rangle_e}{r_0^3}+\mathcal{O}(\sigma^{-3}),
    \end{equation} 
    \begin{equation}\label{he}
        h=\frac{\int_{M_t}Hd\mu_t}{|M_t|}=\oint_{M_t}Hd\mu_t=\frac{2}{r_0}-\frac{4m}{r_0^2}+\frac{6m}{r_0^3}\oint_{M_t}\langle\vec{a}, \nu_e\rangle_e d\mu_t +\mathcal{O}(\sigma^{-3}), 
    \end{equation} 
    \[ \oint_{M_t}H^2d\mu_t=\frac{4}{r_0^2}-\frac{16m}{r_0^3}+\frac{16m^2+24m\oint_{M_t}\langle\vec{a}, \nu_e\rangle_ed\mu_t}{r_0^4}+\mathcal{O}(\sigma^{-4}). \]
    Thus, we get at $(y_0,t_0)$
    \begin{equation}\label{3.2}
    \begin{split}
        &h(H-h_0)=hH-\oint_{M_t}H^2d\mu_t\\
        =&\frac{12m\langle\vec{a}, \nu_e\rangle_e}{r_0^4}-\frac{12m}{r_0^4}\oint_{M_t}\langle\vec{a}, \nu_e\rangle_e+\mathcal{O}(\sigma^{-4})\\
        \geq& \frac{6m|\vec{a}|}{r_0^4}-\frac{12m}{r_0^4}\oint_{M_t}\langle\vec{a}, \nu_e\rangle_e+\mathcal{O}(\sigma^{-4})\\
        \geq& (4m|\vec{a}|-c(C_0^2+B_2+B_3))\sigma^{-4},
        \end{split}
    \end{equation} 
    provided that $\sigma\geq c(C_0^2+B_1^2+B_2+B_3)$. At $(y_0, t_0)$, in view of Proposition \ref{3.4},
     we have
     \[ |\vec{a}|\geq D-c(C_0+B_2+B_3). \]
     Combining with \eqref{3.2}, we finally get
     \[ 0\geq h(H-h_0)\geq (4mD-(4m+1)c(C_0^2+B_2+B_3))\sigma^{-4}, \]
     which yields a contradiction if $D>c(m^{-1}+1)(C_0^2+B_2+B_3)$.
\end{proof}

By a similar argument as Huisken and Yau (Lemma 3.6-Proposition 3.12 in \cite{huisken1996definition}), we can derive curvature estimates for the flow \eqref{flow0}. This implies the long-time existence and smooth convergence of the flow \eqref{flow0}. 
Finally, we will show that the solution surfaces $M_t$ converge exponentially fast to a constant mean curvature surface, highlighting the main difference between the area preserving mean curvature flow (\ref{flow0}) and the flow (\ref{flow-VPMCF}) considered by Huisken and Yau.

We can easily compute that 
\begin{equation}\label{11}
    \int_{M_t}(H-h)^2d\mu_t=h(h_0-h)|M_t|,
\end{equation} 
\begin{equation}\label{22}
    \int_{M_t}(H-h_0)^2d\mu_t=h_0(h_0-h)|M_t|,
\end{equation}
\begin{equation}\label{Ric}
    \bar{Ric}(\nu, \nu)=-\frac{2m}{\sigma^3}+c\sigma^{-4}.
\end{equation}
We can calculate
\[ \frac{d}{dt}\left(\frac{h_0}{h}\right)=\frac{h\frac{d}{dt}h_0-h_0\frac{d}{dt}h}{h^2}=\frac{\frac{d}{dt}h_0\int_{M_t}Hd\mu_t-h_0|M_t|\frac{d}{dt}h}{h^2|M_t|}. \]
Note that
\begin{equation*}
    \begin{split}
        &\frac{d}{dt}h_0\int_{M_t}Hd\mu_t-h_0|M_t|\frac{d}{dt}h\\
        =&\frac{d}{dt}\int_{M_t}H^2d\mu_t-2h_0\frac{d}{dt}\int_{M_t}Hd\mu_t\\
        =& \int_{M_t}[2H(\Delta H+(|A|^2+\bar{Ric}(\nu, \nu))(H-h_0))+H^3(h_0-H)]d\mu_t\\
        &-2h_0\int_{M_t}[(|A|^2+\bar{Ric}(\nu, \nu))(H-h_0)+H^2(h_0-H)]d\mu_t\\
        =& \int_{M_t}[2H\Delta H +2(|A|^2+\bar{Ric}(\nu, \nu))(H-h_0)^2-H^2(H-h_0)^2\\
        &\quad\quad +h_0H^2(H-h_0)]d\mu_t\\
        =& \int_{M_t}[-2|\nabla H|^2 +2(|\A|^2+\bar{Ric}(\nu, \nu))(H-h_0)^2
      +h_0H(H-h_0)^2]d\mu_t.\\
    \end{split}
\end{equation*}
Thus, by Lemma 3.13 in \cite{huisken1996definition} and \eqref{He}, \eqref{11}, \eqref{22}, \eqref{Ric}, we have
\begin{equation}\label{h0h}
    \begin{split}
        &\frac{d}{dt}\left(\frac{h_0}{h}\right)\\
        \leq& \frac{-2}{h^2|M_t|}\left( \frac{2}{\sigma^2}-\frac{4m}{\sigma^3}-cC_0B_1\sigma^{-4} \right)\int_{M_t}(H-h)^2d\mu_t\\
        &+\frac{2}{h^2|M_t|}\left(\frac{-2m}{\sigma^3}+cC_0B_1\sigma^{-4}\right)\int_{M_t}(H-h_0)^2d\mu_t\\
       & +\frac{1}{h^2|M_t|}\left( \frac{4}{\sigma^2}-\frac{16m}{\sigma^3}+-cC_0B_1\sigma^{-4} \right)\int_{M_t}(H-h_0)^2d\mu_t\\
        =& \left(\frac{h_0}{h}-1\right)\left( -\frac{4}{\sigma^2}+\frac{8m}{\sigma^3}+cC_0B_1\sigma^{-4} \right)\\
        &+\frac{h_0}{h}\left(\frac{h_0}{h}-1\right)\left(\frac{4}{\sigma^2}-\frac{20m}{\sigma^3}+cC_0B_1\sigma^{-4}\right)\\
        =& \left(\frac{h_0}{h}-1\right)\left( -\frac{12m}{\sigma^3}+cC_0B_1\sigma^{-4} \right)\\
       & +\left(\frac{h_0}{h}-1\right)^2\left(\frac{4}{\sigma^2}-\frac{20m}{\sigma^3}+cC_0B_1\sigma^{-4}\right).
    \end{split}
\end{equation}
Now notice that $M_t\subset\mathcal{B}_{\sigma}(B_1, B_2, B_3)$ for all time $t\geq0$ such that by (\ref{vol-evo})
\[ \int_0^{\infty}\int_{M_t}(h_0-H)d\mu_tdt=|M_0|\int_0^{\infty}(h_0(t)-h(t))dt\leq \text{Vol}(B_{\sigma+B_1}) \]
is bounded. Since $h_0(t)-h(t)\geq 0$, we see that $h_0(t)-h(t)$ tends to zero as $t\rightarrow\infty$. On the other hand, $h(t)=\mathcal{O}(\sigma^{-1})$ has a uniform lower bound, we see that $\frac{h_0}{h}\to 1$ as $t\rightarrow\infty$. In particular, there is $t_0$ such that 
\[ \frac{h_0}{h}-1\leq \eps\sigma^{-2}, \]
for $t>t_0$.
Therefore, from \eqref{h0h},  
\[ \frac{d}{dt}\left(\frac{h_0}{h}\right)\leq \left(\frac{h_0}{h}-1\right)\left( -\frac{(12-\eps)m}{\sigma^3} \right),  \]
which implies that
\[ \left(\frac{h_0(t)}{h(t)}-1\right)\leq \left(\frac{h_0(0)}{h(0)}-1\right)e^{-\frac{(12-\eps)m}{\sigma^3}t}.  \]
Using the fact that $h(t)=\mathcal{O}(\sigma^{-1})$ again, we obtain that 
\[ 0\leq h_0(t)-h(t)\leq ce^{-\frac{(12-\eps)m}{\sigma^3}t}  \]
for some constant $c$. Inserting this estimate into (\ref{22}) and use the fact that $h_0(t)=\mathcal{O}(\sigma^{-1})$, we finally obtain the exponential decay of the $L^2$-integral
\[ \int_{M_t}(H-h_0)^2d\mu_t\leq c'e^{-\frac{(12-\eps)m}{\sigma^3}t}  \]
for another constant $c'$. Exponential convergence in higher norms are now followed from standard interpolation inequalities, which completes
the proof of Theorem \ref{asy}.
\end{proof}

%\textbf{Conflict of Interest} The authors have no conflict of interest to declare.

%\textbf{Data availability} The authors declare no datasets were generated or analysed during the current study.

\vspace{.2in}

\bibliography{refofmix}
\bibliographystyle{alpha}

\date{November 2024}

\end{document}